\documentclass[11pt]{amsart}

\usepackage{amsmath}
\usepackage{amssymb}
\usepackage{graphicx}
\usepackage[initials]{amsrefs} 

\usepackage[colorlinks]{hyperref} 
\usepackage[capitalise]{cleveref}
\usepackage{amsthm}    
\usepackage{mathrsfs}
\newtheorem{theorem}{Theorem}[section]
\newtheorem{corollary}[theorem]{Corollary}

\newtheorem*{question}{Question} 
\newtheorem{proposition}[theorem]{Proposition}
\theoremstyle{definition}

\numberwithin{equation}{section}

\newcommand{\R}{{\mathbb{R}}}
\newcommand{\eps}{{\varepsilon}}

\renewcommand{\P}{{\mathcal{P}}}
\newcommand{\C}{{\breve{C}}}
\newcommand{\cpw}{{\breve{\breve{w}}}}
\newcommand{\cw}{{\breve w}}
\newcommand{\K}{{\mathcal{K}}}

\newcommand{\cR}{{\mathscr{R}}}

\newcommand{\Sph}{{\mathbb{S}}}

\newcommand{\EE}{{\mathcal{E}}}

\title[Whitney-type estimates for convex functions]{Whitney-type estimates for convex functions}

\author{Jaskaran Singh Kaire}
\address{Department of Mathematics, University of Manitoba, Winnipeg, MB, R3T 2N2, Canada}
\email{singhj82@myumanitoba.ca}
\thanks{The first author was supported by the University of Manitoba Faculty of Science Undergraduate Research Award.}

\author{Andriy Prymak}
\address{Department of Mathematics, University of Manitoba, Winnipeg, MB, R3T 2N2, Canada}
\email{prymak@gmail.com}
\thanks{The second author was supported by NSERC of Canada Discovery Grant RGPIN-2020-05357.}



\begin{document}
	
	\begin{abstract}
		We study Whitney-type estimates for approximation of convex functions in the uniform norm on various convex multivariate domains while paying a particular attention to the dependence of the involved constants on the dimension and the geometry of the domain.
	\end{abstract}	
	
	\maketitle
	
	\section{Introduction and results}
	
	\subsection{Introduction}
	
	Whitney~\cite{Wh} showed that for any function $f$ continuous on $[0,1]$ there exists an algebraic polynomial $p_{m-1}$ of degree $\le m-1$ such that
	\begin{equation}\label{eqn:basic Whitney}
		\max_{x\in[0,1]}|f(x)-p_{m-1}(x)|\le w(m) \max_{x,x+mh\in[0,1]}\left|\sum_{j=0}^m (-1)^j \binom{m}{j}f(x+jh)\right|,
	\end{equation}
	where $w(m)$ is a positive constant depending only on $m$, and $m$ is an arbitrary positive integer. In the inequality~\eqref{eqn:basic Whitney}, the left-hand side is the error of uniform approximation of $f$ by $p_{m-1}$, while the maximum in the right-hand side is the $m$-th order modulus of smoothness. Bounds of the approximation error by a measure of smoothness are classical in approximation theory, see, e.g.~\cite{DeLo}*{Chapters~2, 7}. 
	
	Typically, a Whitney type estimate like~\eqref{eqn:basic Whitney} would be applied on an interval of a partition and used for construction of piecewise polynomial approximation, so the degree $m$ is fixed while the number of pieces grows. In other words, Whitney type inequalities help bound \emph{local} approximation error.
	
	Estimating the values of involved (Whitney) constants is an important question which attracted a lot of attention. As an example of such a result, let us mention that Gilewicz, Kryakin and Shevchuk~\cite{GiKrSh} obtained the best known bound on the smallest possible $w(m)$ in~\eqref{eqn:basic Whitney} valid for all positive integers $m$, which is $w(m)\le 2+e^{-2}$.
	
	In the multivariate settings, a comprehensive study of Whitney constants was conducted by Brudnyi and Kalton in~\cite{BrKa}. For approximation on convex domains, it was shown by Dekel and Leviatan~\cite{DeLe} that the corresponding Whitney constant does not depend on specific geometry of the domain. Recently, the second author and Dai~\cite{DaPr} obtained directional Whitney inequalities valid for certain classes of domains which are not necessarily convex. 
	
	Our goal in this work is to establish Whitney-type estimates for the case when the function we need to approximate is \emph{convex}. This additional restriction may, by itself, lead to better approximation rate (smaller values of Whitney constant). Another and different problem of our interest is to study Whitney-type estimates for approximation of convex function by polynomials which are also required to be convex, i.e., the problem of \emph{shape preserving approximation}. For a survey of shape preserving polynomial approximation on an interval (in global settings), see~\cite{KLPS}. Very little is known for convexity preserving multivariate polynomial approximation, where perhaps the most significant contribution is that of Shvedov~\cite{Sh-multivar}. 
	
	In our investigations in this work, we made an effort to track and decrease the dependence of involved constants on the dimension of the space and the geometry of the domain, which could be of particular importance for possible application in data science where one needs to work with high dimensional data. 
	
	\subsection{Notations}
	
	Let $\P_{m,n}$ denote the space of algebraic polynomials of total degree $\le m$ in $n$ variables, and $C(K)$ be the space of continuous real valued functions on $K$ equipped with the norm $\|f\|_K:=\max_{x\in K}|f(x)|$, where $K\subset\R^n$ is a compact set. For $f\in C(K)$ the error of uniform polynomial approximation is defined as
	\[
	E_m(f;K):=\inf_{P\in \P_{m,n}}\|f-P\|_K,
	\]
	(note that in some literature, e.g., in~\cite{BrKa}, $\P_{m-1,n}$ is used in place of $\P_{m,n}$ in the definition of $E_m$, but in our opinion matching the index to the total degree of the polynomial is more natural in multivariate settings) and the $m$-th modulus of smoothness of $f$ on $K$ is
	\[
	\omega_m(f;K):=\max_{x,x+h,\dots,x+mh\in K}|\Delta^m_h(f;x)|,
	\]
	where
	\[
	\Delta^m_h(f;x):=\sum_{j=0}^m (-1)^j \binom{m}{j}f(x+jh).
	\]
	One can find elementary properties of moduli of smoothness in~\cite{DeLo}*{Sect.~2.7}. 
	
	We define the \emph{Whitney constant} $w_m(K)$ by
	\begin{equation}\label{eqn:def whitney}
		w_m(K):=\sup\{E_{m-1}(f;K): f\in C(K) \text{ and } \omega_m(f;K)\le 1 \}.
	\end{equation}
	In what follows, we assume that $K\subset \K^n$, where $\K^n$ is the class of $n$-dimensional convex bodies, i.e., compact convex (the segment joining any two points of $K$ entirely belongs to $K$) subsets of $\R^n$ having nonempty interior. Note that any compact convex set in $\R^n$ with empty interior is a convex body in an appropriate affine subspace of $\R^n$ of smaller dimension. We let $\C(K)$ to be all functions from $C(K)$ which are convex on $K$, i.e. those $f\in C(K)$ satisfying $f(x)+f(y)\ge2 f((x+y)/2)$ for any $x,y\in K$. Now we can define \emph{Whitney constant for convex functions} $\cw_m(K)$ by
	\begin{equation}\label{eqn:def whitney for convex}
		\cw_m(K):=\sup\{E_{m-1}(f;K): f\in \C(K) \text{ and } \omega_m(f;K)\le 1 \}.
	\end{equation}
	For $f\in\C(K)$, the error of approximation by \emph{convex} polynomials is
	\[
	\breve E_m(f;K):=\inf_{P\in \P_{m,n}\cap \C(K)}\|f-P\|_K.
	\]
	Finally, the \emph{convexity preserving Whitney constant} $\cpw_m(K)$ is
	\begin{equation}\label{eqn:def convexity pres whitney}
		\cpw_m(K):=\sup\{\breve E_{m-1}(f;K): f\in \C(K) \text{ and } \omega_m(f;K)\le 1 \}.
	\end{equation}
	The following relation between the three constants defined in~\eqref{eqn:def whitney}, \eqref{eqn:def whitney for convex} and \eqref{eqn:def convexity pres whitney} is immediate:
	\begin{equation}\label{eqn:obvious relation between whitneys}
		w_m(K)\ge \cw_m(K)\le \cpw_m(K).
	\end{equation} 
	We define the corresponding \emph{global} Whitney constants $w_{m,n}$, $\cw_{m,n}$ and $\cpw_{m,n}$ as the suprema of the left-hand-sides of~\eqref{eqn:def whitney}, \eqref{eqn:def whitney for convex} and \eqref{eqn:def convexity pres whitney}, respectively, over $K\in \K^n$.
	
	It is straightforward that for any $f\in C(K)$ we have $E_0(f;K)=\frac12 \omega_1(f;K)$ and $w_1(K)=\cw_1(K)=\cpw_1(K)=\frac12$, so we will proceed by discussing the cases when $m\ge 2$.
	
	\subsection{Approximation by linear functions}
	One of the main results of~\cite{BrKa} is the next theorem.
	\begin{theorem}[\cite{BrKa}*{Th.~3.1}]\label{thm:brudnyi kalton linear general}
		The following inequalities hold:
		\[
		\frac12\log_2\left(\left[\frac n2\right]+1\right) 
		\le w_{2,n}
		\le \frac12[\log_2 n]+\frac 54,
		\]
		in particular,
		\[
		\lim_{n\to\infty}\frac{w_{2,n}}{\log_2n}=\frac12.
		\]
	\end{theorem}
	($[x]$ denotes the largest integer not exceeding $x$.)
	
	Since any element of $\P_{1,n}$ is a convex function, $\cw_{2,n}(K)=\cpw_{2,n}(K)$, so we will be only concerned with $\cw_2(K)$. By minor modifications of the proofs from~\cite{BrKa}, we will show that the behaviour of the corresponding Whitney constant for convex functions is smaller by essentially the factor of $\frac12$, namely, we prove:
	\begin{theorem}\label{thm:convex linear general}
		The following inequalities hold:
		\[
		\frac14\log_2(n+1) 
		\le \cw_{2,n}
		\le \frac14[\log_2 n]+\frac 34,
		\]
		in particular,
		\[
		\lim_{n\to\infty}\frac{\cw_{2,n}}{\log_2n}=\frac14.
		\]
	\end{theorem}
	The lower bound in~\cref{thm:brudnyi kalton linear general} is obtained by considering $K$ as the Cartesian product of two simplexes, while in~\cref{thm:convex linear general} we simply take $K$ as a simplex. Either way, the ``bad'' domains $K$ are not centrally symmetric. However, it is known that asymptotically the behavior of the global Whitney constant is the same up to an absolute constant factor $C$ (independent of dimension) even if one restricts the domains to be symmetric. Namely (see~\cite{BrKa}*{Remark~(c), p.~162})
	\[
	w_{2,n}^{(s)}\le w_{2,n} \le C w_{2,n}^{(s)},
	\]
	where $w_{m,n}^{(s)}=\sup\{w_m(K): K\in\K^n, K=-K\}$. In particular, we have $\lim_{n\to\infty} w_{2,n}^{(s)}=\infty$.
	
	We show that the situation for Whitney constants for convex functions is completely different when the domain is symmetric. Moreover, we find the exact value of the corresponding constant for arbitrary centrally symmetric convex domain.
	\begin{theorem}\label{thm:convex linear symmetric}
		For arbitrary $K\in\K^n$, $K=-K$, we have $\cw_{2}(K)=\frac12$.
	\end{theorem}
	The main idea of the upper bound is to utilize the existence of a supporting hyperplane at the center of symmetry to the graph of the function which needs to be approximated. The lower bound is rather standard and is essentially one-dimensional.
	
	\subsection{Approximation of convex functions by polynomials of degree $\ge2$}
	In contrast to the previous subsection, once the degree of the approximating polynomial is at least 2, then one does not get any improvement in the values of Whitney constants even if the function to be approximated is convex. Our main result for this situation is the following.
	\begin{theorem}
		\label{thm:convex by at least quadratic}
		For any $K\in\K^n$ and $m\ge 3$ we have $\cw_m(K)=w_m(K)$.
	\end{theorem}
	The key idea for the proof is that for any sufficiently smooth function one can add an appropriate quadratic polynomial to that function to make it convex.
	
	Using~\cref{thm:convex by at least quadratic} and the corresponding results on the usual Whitney constant $w_m(K)$, one can obtain direct corollaries regarding the behaviour of $\cw_m(K)$, when $m\ge3$. For the global Whitney constants, the following is conjectured: if $m\ge 2$, then for any $n$
	\begin{equation}\label{eqn:conjecture for global}
		w_{m,n}\approx w_{m,n}^{(s)}\approx n^{m/2-1}\log(n+1),
	\end{equation}
	where the constants in the equivalences may depend on $m$ only. As we have seen from the previous subsection, this is confirmed for $m=2$. For $m\ge 3$ only the upper estimate on $w_{m,n}^{(s)}$ is known, as well as the following lower bounds: $w_{m,n}\ge w_{m,n}^{(s)}\ge c\sqrt{n}$. An interested reader is referred to~\cite{BrKa} for details and other known results for specific domains, such as unit balls in $l_p$ metric with $1\le p\le \infty$.
	
	\subsection{Convexity preserving approximation by polynomials of degree $\ge2$}
	
	We begin with a negative result for $m\ge4$.
	\begin{theorem}
		\label{thm:negative convexity preserving}
		For any $K\in\K^n$, $m\ge 4$, we have $\cpw_{m}(K)=\infty$.
	\end{theorem}
	The proof readily follows from the one-dimensional version due to Shvedov~\cite{Sh-orders}*{Th.~3}.
	
	Thus, we are left with the study of $\cpw_{3}(K)$. 
	
	In the one-dimensional case, one can observe that if $f\in \C(I)$ for a segment $I\subset\R$, then the best approximating quadratic is automatically convex, i.e. $\breve E_2(f;I)=E_2(f;I)$. Indeed, if $P\in \P_{2,1}$ is the best approximant to $f$, i.e., $\|f-P\|_{I}=E_2(f;I)$, then $P''$ is constant and either $P''\ge0$ in which case $P\in\C(I)$ so there is nothing to prove, or $P''<0$ and then $f-P$ is strictly convex on $I$ and by the properties of convex functions one can find a linear function $L$ which is between $f$ and $P$ on $I$, so $L$ would satisfy $\|f-L\|_I<\|f-P\|_I$, contradiction.  So, for any segment $I\subset\R$, 
	\begin{equation}\label{eqn:cpw3 segment}
		\cpw_{3}(I)=\cw_3(I)\le w_3(I)\le 1,
	\end{equation}
	where the first inequality uses~\eqref{eqn:obvious relation between whitneys} and the last inequality was obtained by Kryakin in~\cite{Kr}.
	
	For the multivariate case, we will show how any quadratic approximating polynomial to a convex function may be modified to become convex quadratic polynomial so that the error of approximation increases by at most an extra constant factor that depends on the geometry of the domain, namely, on how far the domain is from being an ellipsoid. 
	\begin{theorem}
		\label{thm:fixing quadratic to be convex}
		For any $K\in\K^n$, $f\in\C(K)$ and $P\in\P_{2,n}$ then there exists $Q\in\P_{2,n}\cap \C(K)$ such that 
		\[
		\|f-Q\|_K \le a(K) \|f-P\|_K, \quad\text{where}\quad a(K):=2(d(K))^2
		\]
		and $d(K)$ is the Banach-Mazur distance between $K$ and the unit ball $B_2^n$ of $\R^n$ defined as
		\[
		d(K):=\inf_{A\in\mathcal{M}_{n\times n}(\R):\det A\ne0,\,b\in\R^n}\{\lambda:B_2^n\subset AK+b\subset \lambda B_2^n\}.
		\]
		In particular, for any $K\in\K^n$, $f\in\C(K)$, we have 
		\[
		\breve E_2(f;K) \le a(K) E_2(f;K)\quad\text{and}\quad \cpw_3(K)\le a(K) \cw_3(K).
		\]
	\end{theorem}
	Recall that $\cw_3(K)\le w_3(K)$ by~\eqref{eqn:obvious relation between whitneys}, so any upper bound on either of $\cw_3(K)$ or $w_3(K)$ yields an upper bound on $\cpw_3(K)$ with the extra factor of $a(K)$, which, even for general convex domain $K$ can be shown to grow at most as $O(n^2)$. Namely, using the known results on the Banach-Mazur distance (which are corollaries of John's characterization of inscribed ellipsoid of largest volume) and estimates on Whitney constants from~\cite{BrKa}, we obtain the following corollary.
	\begin{corollary}
	\label{cor:specific for quadratics} There exists $c>0$ such that for all $n$ we have
	\begin{align}
		\cpw_3(B^n_2)&\le c \log(n+1), \label{eqn:cor-ball} \\
		\cpw_3(K)&\le 2 n w_3(K) \le c n^{3/2} \log(n+1) \quad\text{for any centrally} \label{eqn:cor-sym} \\ & \hskip6cm\nonumber\text{symmetric }K\in\K^n,  \\
		\cpw_3(K)&\le 2n^2 w_3(K) \quad\text{for any }K\in\K^n. \label{eqn:cor-gen}
	\end{align}
	
\end{corollary}
	
	We note that for all dimensions $n\ge 2$, generally speaking, there is no uniqueness of best uniform approximating polynomial of degree $2$, see, e.g.~\cite{Sha}*{Ch.~2}. We show that unlike in the one-dimensional case, in several variables the set of best approximating quadratics to a convex function may contain a \emph{non-convex} quadratic.
	\begin{proposition}\label{prop:example}
		For $(x,y)\in[-1,1]\times[0,1]=:K$, define
		\[
		f(x,y):=2\max\{1-y,|x|\}.
		\]
		Then $f\in \C(K)$ and $E_2(f;K)=\frac 12$. Denote $\EE:=\{R\in\P_{2,2}:\|f-R\|_K=E_2(f;K)\}$ to be the set of best quadratic approximations to $f$ on $K$. Then: \\
		(i)
		$
		P(x,y):=\frac32+x^2-y^2\in \EE
		$
		and $P\notin \C(K)$;\\
		(ii) $Q(x,y):=\frac32+x^2+y^2-2y\in\EE$ and $Q\in \C(K)$.
	\end{proposition}
	
	Since in our example there is also a \emph{convex} best approximating quadratic, one can still hope that for any convex function it is always possible to \emph{choose} a best approximating quadratic which is convex itself. 
	\begin{question}
		Is it true that for any $K\in\K^n$, $n\ge 2$, and for any $f\in\C(K)$, it is possible to choose $P\in\P_{2,n}\cap \C(K)$ such that $\|f-P\|=E_2(f;K)$?
	\end{question}
	The affirmative answer would imply that \cref{thm:fixing quadratic to be convex} is valid with $a(K)=1$. A more accessible question could be to find out if the statement of \cref{thm:fixing quadratic to be convex} is valid with $a(K)<c$ for a constant $c$ independent of $n$ and $K$.

	\section{Proofs}
	
	\subsection{Proof of \cref{thm:convex linear general}} It is possible to compute $E_1(f;K)$ for $K\in \K^n$ using the following result, which, in a certain sense, is a generalization of the Chebyshev alternation theorem to the multivariate case for approximation by linear polynomials. 
	\begin{proposition}[\cite{BrKa}*{Prop.~3.2}]
		\label{prop:E_1 computation general}
		For any $f\in C(K)$, $K\in\K^n$, we have
		\begin{equation}\label{eqn:E_1 comp general}
			E_1(f;K)=\frac12\max \left\{ \sum_{i=1}^la_i f(x_i)-\sum_{j=1}^m b_j f(y_j) \right\},
		\end{equation}
		where the maximum is taken over all positive integers $l$, $m$ with $l+m\le n+2$, all subsets $\{x_i\}$, $\{y_j\}$ of $K$ and nonnegative coefficients $\{a_i\}$, $\{b_j\}$ satisfying
		\begin{equation}\label{eqn:convex hulls intersect}
			\sum_{i=1}^la_i=\sum_{j=1}^m b_j=1 \quad\text{and}\quad
			\sum_{i=1}^la_ix_i=\sum_{j=1}^m b_jy_j.
		\end{equation}
	\end{proposition}
	When $f$ is convex, the computation is easier, and it suffices to take $l=n+1$ and $m=1$.
	\begin{corollary}
		\label{cor:E_1 computation convex}
		For any $f\in \C(K)$, $K\in\K^n$, we have
		\begin{equation}\label{eqn:E_1 comp convex}
			E_1(f;K)=\frac12\max \left\{ \sum_{i=1}^{n+1}a_i f(x_i)-f \left( \sum_{i=1}^{n+1} a_i x_i \right) \right\},
		\end{equation}
		where the maximum is taken over all subsets $\{x_i\}$ of $K$ and nonnegative coefficients $\{a_i\}$ satisfying $\sum_{i=1}^{n+1}a_i=1$.
	\end{corollary}
	\begin{proof}
		Clearly, the right-hand side of~\eqref{eqn:E_1 comp general} is greater than or equal to the right-hand side of~\eqref{eqn:E_1 comp convex}. Taking arbitrary $\{x_i\}$, $\{y_j\}$, $\{a_i\}$, $\{b_j\}$ satisfying~\eqref{eqn:convex hulls intersect} and using Jensen's inequality, we have
		\begin{align*}
	\sum_{i=1}^la_i f(x_i)-\sum_{j=1}^m b_j f(y_j)  & \le \sum_{i=1}^la_i f(x_i)-f \left(\sum_{j=1}^m b_j y_j\right) \\ &= \sum_{i=1}^la_i f(x_i)-f \left(\sum_{i=1}^l a_i x_i\right),
		\end{align*}
		so the inequality in the other direction follows (one can take $a_i=0$ with arbitrary $x_i\in K$ for $l+1\le i\le n+1$).
	\end{proof}
	Following~\cite{BrKa}*{Sect.~3}, we denote	for $f\in C(K)$, $K\in\K^n$
	\begin{equation*}
	\delta_{n+1}(f):=\max\left|\sum_{i=1}^{n+1}a_i f(x_i)-f \left( \sum_{i=1}^{n+1} a_i x_i \right)\right|,
	\end{equation*}
	where the maximum is taken over all subsets $\{x_i\}$ of $K$ and nonnegative coefficients $\{a_i\}$ satisfying $\sum_{i=1}^{n+1}a_i=1$, and
	\[
	\alpha_n(K):=\sup\{\delta_{n+1}(f):f\in C(K),\ \omega_2(f;K)\le 1\}.
	\]
	It is straightforward to observe that $\alpha_n(K)\le \sup_{S}\alpha_n(S)$, where the supremum is taken over all $n$-dimensional simplexes $S\subset K$, i.e., $S$ which are convex hulls of $n+1$ points in $K$ with interior of $S$ being nonempty. Denote by $S^n$ some fixed $n$-dimensional simplex. Now any $n$-dimensional simplex $S$ can be mapped into $S^n$ by a nonsingular affine transform, implying $\alpha_n(S)=\alpha_n(S^n)$. In summary, $\alpha_n(K)\le \alpha_n(S^n)=:\beta_n$, therefore, by~\cref{cor:E_1 computation convex}, 
	\begin{equation}\label{eqn:cw2 bound through beta}
		\cw_{2,n}\le\frac12\beta_n.
	\end{equation}
	(One can compare with~\cite{BrKa}*{Eqn.~(3.2)} for the corresponding usual Whitney constant.) By~\cite{BrKa}*{Lemma~3.4}, $\beta_{2r}\le \beta_r+\frac12$ for all positive integers $r$. As $\beta_1=1$, this implies $\beta_{2^r}\le \frac12r+1$. Taking $r=[\log_2n]+1$ we obtain from~\eqref{eqn:cw2 bound through beta} the upper bound in \cref{thm:convex linear general}.
	
	The proof of the lower bound in \cref{thm:convex linear general} is essentially given on~\cite{BrKa}*{p.~177} as the function $f_n$ defined by~\cite{BrKa}*{Eqn.~(3.4)} is shown there to be convex and satisfy $\omega_2(f_n;S^n)\le 1$ and $E_1(f_n;S^n)\ge\frac14\log_2(n+1)$. Here we will restrict ourselves only to repeating the definition of $f_n$. We can consider $S^n$ as a subset of $\R^{n+1}$ consisting of $x=(x_1,\dots,x_{n+1})$ such that each coordinate is nonnegative and $\sum_{k=1}^{n+1}x_k=1$. Then $f_n\in C(S^n)$ is defined by
	\[
	f_n(x)=\frac12\sum_{k=1}^{n+1}x_k\log_2 x_k.
	\]
	Now we are done with the proof of \cref{thm:convex linear general}.
	
	\subsection{Proof of \cref{thm:convex linear symmetric}} Suppose $K\in\K^n$, $K=-K$. 
	
	We begin with showing that $\cw_2(K)\le \frac12$. We need certain preliminaries from the theory of multivariate convex functions. Suppose $U\subset\R^n$ and $f:U\to\R$ is convex. $f$ is said to have support at $x_0\in U$ if there exists $l\in\P_{1,n}$ such that $l(x_0)=f(x_0)$ while $l(x)\le f(x)$ for any $x\in U$. If $U$ is open, then $f$ has support at any point in $U$, see, e.g.~\cite{RoVa}*{Th.~B, p.~108}. Consequently, any convex and continuous function on $K\in\K^n$ has support at any interior point of $K$. Now let $f\in \C(K)$ be arbitrary. By the above, $f$ has support at $0$, so we can find $l\in\P_{1,n}$ satisfying $l(0)=f(0)$ and $l(x)\le f(x)$ for any $x\in K$. Set $g(x):=f(x)-l(x)$, $x\in K$. Then $g$ is convex nonnegative on $K$ function with $g(0)=0$, $E_1(g;K)=E_1(f;K)$ and $\omega_2(g;K)=\omega_2(f;K)$. Therefore, it suffices to show that $E_1(g;K)\le \frac12 \omega_2(g;K)$. Indeed, define $P(x):=\frac12\|g\|_K$. Then
	\begin{equation}\label{eqn:E1 through g}
		E_1(g;K)\le \|g-P\|_K=\frac12 \|g\|_K.
	\end{equation}
	Let $x^*\in K$ be a point such that $\|g\|_K=g(x^*)$. Then $-x^*\in K$ and
	\begin{align} \nonumber
	\omega_2(g;K) & \ge |\Delta^2_{x^*}(g;-x^*)|=g(x^*)-2g(0)+g(-x^*)=g(x^*)+g(-x^*)\\&\ge g(x^*)=\|g\|_K. \label{eqn:omega2 from below}
\end{align}
	Now~\eqref{eqn:E1 through g} and~\eqref{eqn:omega2 from below} imply $E_1(g;K)\le \frac12 \omega_2(g;K)$, and so $\cw_2(K)\le \frac12$.
	
	For any $\eps>0$ we will show that $\cw_2(K)>\frac12-\eps$. Since $\cw_2(K)$ does not change if we apply a rotation and/or a dilation to $K$, we can assume, without loss of generality, that the projection of $K$ onto the first coordinate axis is precisely the segment $[-1,1]$. For $\delta\in(0,1)$ which will be selected later, let 
	\[
	f_\delta(x_1,\dots,x_n)=\max\left\{0,\frac{x_1-1+\delta}{\delta}\right\}=:g_\delta(x_1).
	\]
	Then obviously $f_\delta\in \C(K)$. Since $f_\delta$ depends only on the first variable, due to our choice of the position of $K$, it is not hard to see that $E_1(f_\delta;K)=E_1(g_\delta;[-1,1])$ and $\omega_2(f_\delta;K)=\omega_2(g_\delta;[-1,1])$. The Chebyshev alternation theorem (e.g.~\cite{DeLo}*{Th.~5.1, p.~74}) implies that $p(x)=\frac x2 + \frac \delta4$ is the best uniform approximation to $g_\delta$ on $[-1,1]$ and $E_1(g_\delta;[-1,1])=\frac12-\frac\delta4$. On the other hand, due to convexity of $g_\delta$ (or by direct verification), we have $\omega_2(g_\delta;[-1,1])=\Delta^2_1(g;-1)=1$. Thus, \[\cw_2(K)\ge \frac{E_1(g_\delta;[-1,1])}{\omega_2(g_\delta;[-1,1])}\ge \frac12-\frac\delta4\] and taking $\delta<4\eps$ completes the proof.
	
	\subsection{Proof of \cref{thm:convex by at least quadratic}} Let $\eps>0$ be arbitrary fixed. Then one can find $f\in C(K)$ with 
	\begin{equation}\label{eqn:w_m def}
		E_{m-1}(f;K)\ge w_m(K)-\eps \quad\text{and}\quad \omega_m(f;K)=1.
	\end{equation}
	By the Weierstrass approximation theorem, there exists an algebraic polynomial $g$ such that $\|f-g\|_{K}\le \frac{\eps}{2^m}$. Then~\eqref{eqn:w_m def} implies
	\begin{equation} \label{eqn:error g lower}
		E_{m-1}(g;K)  \ge E_{m-1}(f;K)-E_{m-1}(f-g;K)
		\ge w_m(K)-\eps-\frac{\eps}{2^m}>w_m(K)-2\eps 
	\end{equation}
	and
	\begin{equation}\label{eqn:mod g upper}
		\omega_m(g;K)\le \omega_m(f;K)+\omega_m(g-f;K)\le 1+2^m\|f-g\|_K\le 1+ \eps.
	\end{equation}
	Now we consider $h(x):=g(x)+L\|x\|^2$, where $\|x\|$ is the Euclidean norm of $x\in\R^n$. We claim that with sufficiently large $L>0$ the resulting function $h$ will be convex on $K$. Indeed, it suffices to ensure that $h$ is convex along any segment that belongs to $K$, which, in turn, holds true provided all the second directional derivatives of $h$ are non-negative, i.e.,
	\begin{equation}\label{eqn:second-der}
		\frac{\partial^2 }{\partial \xi^2}h(x)\ge 0 \quad\text{for any }x\in K\quad\text{and}\quad \xi\in\Sph^{n-1},
	\end{equation}
	where $\Sph^{n-1}$ is the unit sphere in $\R^n$. Thus, as $g$ is $C^2$ everywhere and $K$ is compact, we can define
	\[
	L:=\frac12 \max_{x\in K} \max_{\xi\in \Sph^{n-1}} \left|\frac{\partial^2 }{\partial \xi^2} g(x)\right|<\infty,
	\] 
	and~\eqref{eqn:second-der} holds, establishing that $h$ is convex on $K$. 
	Therefore, as $L\|x\|^2$ is a quadratic polynomial and $m\ge 3$, from~\eqref{eqn:error g lower} and~\eqref{eqn:mod g upper}, we conclude
	\[
	\cw_m(K)\ge \frac{E_{m-1}(h;K)}{\omega_m(h;K)}=\frac{E_{m-1}(g;K)}{\omega_m(g;K)} \ge \frac{w_m(K)-2\eps}{1+\eps}.
	\]
	This implies $\cw_m(K)\ge w_m(K)$ and the inequality in the other direction is given in~\eqref{eqn:obvious relation between whitneys}.
	
	\subsection{Proof of \cref{thm:negative convexity preserving}}
	Since $\cpw_m(K)$ is invariant under dilations and translations of $K$, we can assume that the projection of $K$ onto the first coordinate axis is exactly $[0,1]$. By~\cite{Sh-orders}*{Th.~3}, for any $A>0$ there exists $g\in \C([0,1])$ such that $\breve E_{m-1}(g;[0,1])>A$ while $\omega_4(g;[0,1])=1$. Defining $f(x_1,\dots,x_n)=g(x_1)$, we obtain
	\[
	\cpw_m(K)\ge \frac{E_{m-1}(f;K)}{\omega_m(f;K)} = \frac{E_{m-1}(g;[0,1])}{\omega_m(g;[0,1])} \ge \frac{E_{m-1}(g;[0,1])}{2^{m-4}\omega_4(g;[0,1])}>
	\frac{A}{2^{m-4}},
	\]
	implying the required $\cpw_m(K)=\infty$.
	
	\subsection{Proof of \cref{thm:fixing quadratic to be convex}} Using an affine change of variables if needed, by the definition of $d(K)$ we can assume that
	\begin{equation}\label{eqn:K position}
		B_2^n \subset K \subset d(K) B_2^n.
	\end{equation}
	We can write $P(x)=x^\intercal M x + L(x)$ for some symmetric $n\times n$ matrix $M$ and $L\in\P_{1,n}$. By standard linear algebra, there is an orthogonal matrix $O$ such that $M=ODO^{-1}$ for a diagonal matrix $D$. Thus, under the orthogonal change of variables $y=O^{-1}x$, we have $P(Oy)=y^\intercal Dy+L(Oy)$. By $D_+$ we denote the matrix obtained from $D$ by replacing all negative (diagonal) entries with zeroes. We define the required polynomial $Q$ as follows:
	\begin{equation}\label{eqn:Q def}
		Q(x):=(O^{-1}x)^\intercal D_+ (O^{-1}x) + L(x) - \|f-P\|_{B_2^n}, \quad x\in K.
	\end{equation}
	It is easy to see that $Q$ is convex: if $(d_i)_{i=1}^n$ are the diagonal entries of $D$, then
	\[
	y^\intercal D_+y=\sum_{i=1}^n \max\{d_i,0\} y_i^2,
	\]
	which is the sum of convex functions $y_i^2$ with nonnegative coefficients.
	
	First step in establishing the required bound on $\|f-Q\|_K$ is to show
	\begin{equation}\label{eqn:est on ball}
		\|P-Q\|_{B_2^n}\le \|f-P\|_{B_2^n}.
	\end{equation}
	For any $y\in B_2^n$ we have $y^\intercal D y \le y^\intercal D_+ y$, therefore,
	\begin{equation}\label{eqn:upper on f-Q}
		(P(Oy)-Q(Oy))\le \|f-P\|_{B_2^n}.
	\end{equation}
	Next we will obtain an estimate in the other direction. Observe that for any $y\in B_2^n$ we have $\pm Oy\in B_2^n$ since $O$ is orthogonal, so~\eqref{eqn:K position} implies $-Oy,0,Oy\in K$. Now, using convexity of $f$, an elementary upper bound on the second difference, and the representation $P(Oy)=y^\intercal Dy+L(Oy)$, we have
	\begin{equation*}
		0\le \Delta^2_{Oy}(f;-Oy)=\Delta^2_{Oy}(f-P;-Oy)+\Delta^2_{Oy}(P;-Oy)\le 4\|f-P\|_{B_2^n}+2y^\intercal Dy,
	\end{equation*}
	which implies
	\begin{equation}\label{eqn:aux bound}
		-2\|f-P\|_{B_2^n}\le y^\intercal Dy.
	\end{equation}
	For any $y\in B_2^n$ let us define
	\[
	y_-:=(u_1,\dots,u_n), \quad\text{where}\quad u_i=\begin{cases}
		y_i, & \text{if }d_i<0,\\
		0, & \text{otherwise}.
	\end{cases}
	\]
	Then $y_-^\intercal Dy_-=y^\intercal Dy-y^\intercal D_+y$, so from~\eqref{eqn:aux bound} for $y_-$ we obtain
	\[
	P(Oy)-Q(Oy)=y_-^\intercal Dy_-+\|f-P\|_{B_2^n}\ge-\|f-P\|_{B_2^n},
	\]
	hence, in combination with~\eqref{eqn:upper on f-Q}, we get the claimed~\eqref{eqn:est on ball}.
	
	Second, observing that $P(x)-Q(x)-\|f-P\|_{B_2^n}$ is homogeneous of degree $2$ function in $x$, we have for any $1\le \lambda\le d(K)$ and any $x\in B_2^n$ that
	\begin{align*}
		|P(\lambda x)-Q(\lambda x)| &=|\lambda^2(P(x)-Q(x))+(1-\lambda^2)\|f-P\|_{B_2^n}| \\
		&\le \lambda^2\|P-Q\|_{B_2^n} + (\lambda^2-1)\|f-P\|_{B_2^n} \\
		& \le (2\lambda^2-1) \|f-P\|_{B_2^n},
	\end{align*}
	where in the last step we used~\eqref{eqn:est on ball}.
	So, by~\eqref{eqn:K position},
	\begin{equation*}
		\|P-Q\|_K\le (2(d(K))^2-1) \|f-P\|_{B_2^n}.
	\end{equation*}
	Using the last inequality, we complete the proof as follows:
	\begin{align*}
		\|f-Q\|_K &\le \|f-P\|_K + \|P-Q\|_K \\
		& \le \|f-P\|_K+(2(d(K))^2-1) \|f-P\|_{B_2^n} \\
		& \le 2(d(K))^2\|f-P\|_K.
	\end{align*}
	
	\subsection{Proof of \cref{cor:specific for quadratics}} The first step in each of~\eqref{eqn:cor-ball}--\eqref{eqn:cor-gen} is an application of \cref{thm:fixing quadratic to be convex}. We have $d(B_2^n)=1$, so~\eqref{eqn:cor-ball} follows from~\cite{BrKa}*{Th.~4.3(a)} for $p=2$. John's characterization of inscribed ellipsoid of largest volume implies (see, e.g.~\cite{Schn}*{Th.~10.12.2, p.~588}) that $d(K)\le n$ for any $K\in\K^n$, while $d(K)\le \sqrt{n}$ if, in addition, $K$ is centrally symmetric. 
	This immediately yields~\eqref{eqn:cor-gen}, while for~\eqref{eqn:cor-sym} we invoke~\cite{BrKa}*{Th.~4.1}.
	
	\subsection{Proof of \cref{prop:example}} $f$ is convex as the maximum of three linear functions. Next, we need a sufficient condition for a polynomial from $\P_{2,2}$ to be in the set $\EE$ of polynomials of best approximation to $f$. It is known (see, e.g., \cite{Sha}*{Th.~2.3.2, p.~14}) that $R\in\EE$ if one can find a finite set $\cR=\{x_1,\dots,x_r\}\subset K$ and positive constants $c_1,\dots,c_r$ satisfying $|f(x)-R(x)|=\|f-R\|_K$ for any $x\in \cR$, and
	\begin{equation}\label{eqn:best approx condition}
		\sum_{i=1}^r c_i (f(x_i)-R(x_i)) S(x_i)=0 \quad\text{for any}\quad S\in \P_{2,2}.
	\end{equation}
	
	
	(i) $P$ is clearly not convex. It is a straightforward multivariate calculus exercise to show that $\|f-P\|_{K}=\frac12$ and with $\cR:=\{-1,0,1\}\times \{0,1\}$
	\[
	f(x,y)-P(x,y)=(-1)^{x+y}\frac12 \quad\text{for any }(x,y)\in\cR.
	\]
	Direct verification also shows that any $S\in\P_{2,2}$ satisfies
	\[
	S(-1,1)+S(0,0)+ S(1,1)-S(-1,0)-S(0,1)-S(1,0)=0,
	\]
	thus~\eqref{eqn:best approx condition} is satisfied for $P$ with $c_1=\dots=c_6=1$ establishing that $P\in\EE$.
	
	(ii) $Q$ is obviously convex. One can observe that the values of $Q$ and $P$ on $\cR$ coincide. The rest of the proof follows the same lines as that for the part (i).
	
	{\bf Acknowledgment.} The authors thank Andr\'as Kro\'o for making them aware of~\cite{Sha}*{Th.~2.3.2} which allowed to simplify the proof of~\cref{prop:example}, as well as the referee for the useful comments.

\end{document}